\documentclass[12pt, reqno]{amsart}
\usepackage[UKenglish]{babel}
\usepackage[T1]{fontenc}
\usepackage[utf8x]{inputenc}
\usepackage{latexsym,amsfonts,amsmath,amsthm,amssymb,mathrsfs}
\usepackage{xcolor,bm, bbm}
\usepackage{paralist}
\usepackage{todonotes}
\usepackage{dsfont}
\usepackage{float}
\usepackage{caption}
\usepackage{xfrac}
\usepackage[colorlinks, linkcolor=blue,citecolor=blue, anchorcolor=blue,urlcolor=blue,hypertexnames=false]{hyperref}
\usepackage{bm}

\DeclareMathOperator{\R}{\mathbb{R}}
\DeclareMathOperator{\N}{\mathbb{N}}

\def \bk {{\mathbf k}}
\def \br {{\mathbf r}}
\def \bt {{\mathbf t}}
\def \bx {{\mathbf x}}

\def \alp {{\alpha}}

\def \eps {{\varepsilon}}
\def \sig {{\sigma}}

\def \balp {{\boldsymbol \alpha}}

\def \bzero {{\boldsymbol 0}}
\def \bone {{\boldsymbol 1}}

\def \bR {{\mathbb R}}
\def \bZ {{\mathbb Z}}
\def \bN {{\mathbb N}}

\def \cH {{\mathcal H}}
\def \cL {{\mathcal L}}
\def \cR {{\mathcal R}}

\newtheorem{thm}{Theorem}%[section]

\newtheorem{conj}[thm]{Conjecture}
\newtheorem{cor}[thm]{Corollary}

\newtheorem{rem}[thm]{Remark}
\newtheorem{pro}[thm]{Proposition}
\newtheorem{lemma}[thm]{Lemma}
\newtheorem*{rem*}{Remark}
\newtheorem*{thm*}{Theorem}
\newtheorem{example}[thm]{Example}

\title[Twisted Diophantine approximation for matrix maps]{Twisted Diophantine approximation for matrix transformations of tori}

\medskip

\newcommand{\mmod}[1]{{{\,\,\mathrm{mod}\,\,#1}}}

\date{}

\author{Sam Chow}
\address{Sam Chow, Mathematics Institute, Zeeman Building, University of Warwick, Coventry CV4 7AL, UK}
\email{Sam.Chow@warwick.ac.uk}

\author{Qing-Long Zhou}
\address{Qing-Long Zhou, School of Mathematics and Statistics,
Wuhan University of Technology, 430070 Wuhan, PR China}
\email{zhouql@whut.edu.cn}

\begin{document}

%\subjclass{28A80; 11K55 11J83}
%keywords: Twist Diophantine approximation, Shrinking targets, Hausdorff measure. 

\begin{abstract}
Consider a sequence of integral matrices 
$\mathcal{A}=(A_n)_{n\in\N}$, and a $d$-tuple function ${\bf r}=(r_1,\ldots,r_d)\colon \N\to (0,\frac{1}{2})$. For a fixed vector ${\bm \alpha},$  we are interested in the set $\mathcal{T}_{{\bm \alpha}}(\mathcal{A}, {\bf r})$ of vectors ${\bm \beta}\in[0,1)^{d}$ for which $A_n{\bm \alpha}~~\!\!\!\!\!\pmod{1}$ infinitely often lies in the box centred at ${\bm \beta}$, with side lengths $2r_i(n)$ in each coordinate direction. Under mild conditions on $\mathcal{A}$ and  ${\bf r}$, we prove a metric dichotomy for the size of $\mathcal{T}_{{\bm \alpha}}(\mathcal{A}, {\bf r}),$  valid for almost every ${\bm \alpha}$ with respect to any fractal measure with a certain polynomial Fourier decay rate. Furthermore, removing all restrictions on ${\bf r}$, we establish a metric dichotomy for Lebesgue almost every ${\bm \alpha}.$ This solves a variant of a conjecture of Gonz\'{a}lez Robert, Hussain, Shulga and  Ward [Conjecture 1.10, Bull. London Math. Soc. 2025]. Finally, we also establish a Jarn\'{i}k-type theorem for $\mathcal{T}_{{\bm \alpha}}(\mathcal{A}, {\bf r}).$
\end{abstract}

%{\small \tableofcontents}

\maketitle

\section{Introduction}
For an irrational number $\alpha\in\mathbb{R}\backslash\mathbb{Q},$ the sequence $(\{n\alpha\})_{n\in\mathbb{N}}$ of fractional parts is uniformly distributed on $[0,1]$, in contrast to the rational case where the sequence is periodic. In 1900, Minkowski \cite{Minkowski1900} proved that for any irrational $\alpha$ and any $\beta \not \in \mathbb{Z} + \alpha\mathbb{Z}$, the inequality
\begin{equation} \label{Minkowski}
    |\!|n\alpha - \beta|\!|<\frac{1}{4n}
\end{equation}
holds for infinitely many  $n\in\mathbb{N},$ where 
$ |\!| \alpha |\!|$ denotes the distance from $\alpha$ to the nearest integer. The constant $\frac{1}{4}$ was later shown by Khintchine~\cite{Khintchine1946} to be optimal.

Minkowski's theorem above provides a uniform rate of approximation valid for all real numbers. Replacing the right-hand side of (\ref{Minkowski}) by a more rapidly decaying function $\psi: \N \to (0,\frac{1}{2})$ leads naturally to the study of the size of the set
\begin{equation*}
\mathcal{W}(\psi):=\{(\alpha,\beta)\in[0,1)^{2}\colon |\!|n\alpha - \beta|\!|<\psi(n) \text{ i.m. } n\in\N\},
\end{equation*}
where i.m. means `for infinitely many'. Fixing $\beta$ and varying $\alpha$ corresponds to classical \emph{inhomogeneous Diophantine approximation} (with $\beta=0$ giving the homogeneous case), while fixing  $\alpha$ and varying $\beta$ falls within the framework of \emph{twisted Diophantine approximation} \cite{BBDV09,BHKV10}.

\subsection{Homogeneous and inhomogeneous Diophantine approximation}

The $\psi$-well approximable set is defined as
 $$\mathcal{W}_{\beta}(\psi)
 :=\{\alpha\in[0,1)\colon |\!|n\alpha - \beta|\!|<\psi(n) \text{ i.m. } n\in\N\}.$$
The study of the metric properties of $\mathcal{W}_{\beta}(\psi)$ dates back a century to  Khintchine \cite{K24}, who proved for $\beta=0$ that 
\begin{equation*}
\mathcal{L}^{1}\big(\mathcal{W}_{\beta}(\psi)\big)=\begin{cases}
   0,   & \text{if $\sum_{n=1}^{\infty}\psi(n)<\infty$}, \\[2mm]
       1 , & \text{if $\sum_{n=1}^{\infty}\psi(n)=\infty$ and $\psi$ is non-increasing}.
\end{cases}
\end{equation*} 
Here and throughout, we write $\mathcal{L}^{d}(A)$ to denote the Lebesgue measure of  $A\subseteq\mathbb{R}^{d}$. Duffin and Schaeffer \cite{DS41} later demonstrated that the monotonicity condition is essential: they constructed $\psi$ supported on highly composite integers such that
\[
\sum_{n=1}^\infty \psi(n) = \infty, \qquad
\cL^1(\mathcal{W}_0(\psi)) = 0.
\]
Duffin and Schaeffer conjectured that for almost all $\alpha,$ the inequality $|n\alpha-m|<\psi(n)$ has infinitely many coprime solutions $(n,m)$ almost surely if and only if 
$$\sum_{n=1}^{\infty}\frac{\phi(n)}{n}\psi(n)=\infty,$$ 
where $\phi$ is Euler's totient function.  Following important contributions by Gallagher \cite{G61}, Erd\H{o}s \cite{E70}, Vaaler \cite{V78}, and Pollington--Vaughan~\cite{PV90}, the conjecture was solved affirmatively by Koukoulopoulos and Maynard \cite{KM20}. A quantitative version was later given by Aistleitner \emph{et al.} \cite{ABH23},  with a subsequent refinements to the error term by Hauke--Vazquez Saez--Walker \cite{HSW24} and an almost-sharp estimate by Koukoulopoulos--Maynard--Yang \cite{KMY25}.

An inhomogeneous generalisation of Khintchine's theorem was established by Sz\"usz \cite{S58}. 
The non-monotonic scenario is more complicated. Counterexamples to relaxing monotonicity were given in \cite{CHPR25, R17}. Despite rumours of its falsity, the most direct analogue of the Duffin--Schaeffer conjecture remains formally open. The so-called weak Duffin--Schaeffer conjecture also remains open, see \cite{BHV24}. The latter was solved for a wide class of non-monotonic functions in \cite{Cho2018, CT2019, CT24a}. Results with extra divergence were established by Yu \cite{Y19, Y21}.
For systems of linear forms, see \cite{AR2023, Kim25, Y21}. 

Under the assumption that $\psi$ is monotonic, 
Jarn\'{i}k's theorem \cite{J31} implies that
$$\dim_{\rm H}W_{\beta}(\psi)=\frac{2}{\tau+1}, \ \ \text{where } \tau=\liminf_{n\to\infty}\frac{-\log \psi(n)}{\log n}.$$
This result can also be derived from Khintchine's theorem via the mass transference principle of Beresnevich and Velani \cite{BV06}. For a general function $\psi,$ the Hausdorff dimension of  $W_{\beta}(\psi)$ was studied in depth by Hinokuma and Shiga~\cite{HS96}.

\subsection{Twisted Diophantine approximation}

The twisted framework differs from classical inhomogeneous Diophantine approximation. Here, we consider the distribution of the Kronecker sequence 
$(n\alpha \mmod 1)_{n\in\mathbb{N}}$
on the unit interval. The diophantine properties of $\alp$ strongly influence this
distribution \cite{Bec1994, CT}. A natural generalization is to replace the upper bound $\frac{1}{4n}$ in (\ref{Minkowski}) by a general non-increasing function  $\psi\colon \mathbb{N} \to (0,\frac{1}{2}),$ investigating the set
\begin{equation*}
    \mathcal{T}_{\alpha}(\psi):=\left\{ \beta\in [0,1]\colon |\!|n\alpha - \beta|\!|<\psi(n) \quad \text{ i.m. } n\in\mathbb{N} \right\}.
\end{equation*}
In 1955, Kurzweil made a foundational contribution to this question. Recall that an irrational $\alpha$ is \emph{badly approximable} if there exists a constant $c(\alpha)>0$ such that
\begin{equation*}
    |\!|n\alpha|\!| \geq \frac{c(\alpha)}{n} \quad \text{ for all } n\in\mathbb{N}.
\end{equation*}

\begin{thm}[Kurzweil \cite{Kur55}]
\label{kurzweil}
Let $\psi\colon \mathbb{N}\to (0,\frac{1}{2})$ be non-increasing and $\alpha\in\mathbb{R}$ irrational. Then $\mathcal{L}^{1}(\mathcal{T}_{\alpha}(\psi))\in\{0,1\}$. Moreover, if $\alpha$ is badly approximable, then
   \begin{equation*}
\mathcal{L}^{1}(\mathcal{T}_{\alpha}(\psi))=\begin{cases}
   0,   & \text{if $\sum_{n=1}^{\infty}\psi(n)<\infty$}, \\[2mm]
       1 , & \text{if $\sum_{n=1}^{\infty}\psi(n)=\infty$}.
\end{cases}
\end{equation*} 
\end{thm}

Kurzweil also showed that the divergence part of the statement fails whenever $\alp$ is not badly approximable. Nonetheless, 
other irrational $\alp$ can still be studied. Fuchs and Kim \cite[Theorem 1.2]{FuchsKim} refined Kurzweil's theorem by incorporating information about the continued fraction convergents of $\alpha$, and Simmons made further progress in \cite{Simmons15}. When $\mathcal{T}_{\alpha}(\psi)$ has zero Lebesgue measure ---as occurs for $\psi_{\tau}(n)=n^{-\tau}$ with $\tau>1$ --- finer geometric tools are needed to distinguish such null sets. In this direction, Bugeaud \cite[Theorem 1]{B03} and Schmeling--Troubetzkoy \cite[Theorem 3.2]{SchmTrot03} independently obtained the following result.

\begin{thm}[{\cite{B03, SchmTrot03}}]\label{B-th}
For $\alpha\in\mathbb{R}\backslash\mathbb{Q}$ and $\tau>1,$ we have
\begin{equation*}
    \dim_{\mathrm H} \mathcal{T}_{\alpha}(\psi_{\tau})=\frac{1}{\tau},
\end{equation*}
where $\dim_{\mathrm H}$ denotes Hausdorff dimension.
\end{thm}

We note that an early result in this direction was established by Bernik and Dodson \cite[p. 105]{BernikDodson}, who proved a version of the above statement for Lebesgue almost all $\alpha\in[0,1]$.  Subsequent work has extended this result in several directions, including generalisations to wider classes of approximation functions $\psi$ and restricted sets of $\alpha,$ where the exponent $\tau$ in the dimension formula is replaced by the lower order at infinity of $\psi$, see \cite[Theorem 3]{FW06}. Notably, it was also shown in \cite[Theorem 2]{FW06} that there exist irrationals $\alpha$ for which the expected dimension result fails to hold. For more about twisted Diophantine approximation, we refer to \cite{BSV25, CT24b, CZ21, HR25, HW24, KP25}.

\subsection{Dynamical Diophantine approximation}
Let $T_{\alpha}: x \mapsto x+\alpha$ denote the rotation on the unit circle. Minkowski's theorem
studies the approximation properties of the orbit $\{T_{\alpha}^{n}(0)\}_{n\ge1},$ motivating the general problem of quantifying the rate at which orbits in a dynamical system approach a target with a prescribed speed. Let $(X,T,\mu)$ be a measure-preserving dynamical system, where $(X,d)$ is a metric space, $T\colon X\to X$ is a Borel transformation, and $\mu$ is a $T$-invariant Borel probability measure on $X.$ By Birkhoff's ergodic theorem \cite{W82}, when the system is ergodic, the set 
$$\left\{\alpha\in X\colon \liminf_{n\to\infty}d(T^{n}(\alpha),\beta)=0\right\}$$
has full $\mu$-measure for $\mu$-almost every $\beta\in X.$  The gives a qualitative description of orbit distribution, analogous to the density of rationals in $\mathbb{R}$, and naturally leads to quantitative refinements.

The classical \emph{shrinking target problem} in dynamical system $(X,T)$ provides a quantitative refinement of Birkhoff's ergodic theorem. It concerns the size of the set
$$\mathcal{W}_{\beta}(T,(E_n))=\{\alpha\in X\colon T^{n}(\alpha)\in \beta+E_n \text{ for infinitely many } n\in\N\},$$
where $(E_n)_{n\ge 1}$ is a sequence of subsets of $X$ and $\beta\in X.$ This concept, introduced by Hill and Velani \cite{HV95}, generalises the framework of metric Diophantine approximation to a dynamical setting. A direct application of the first Borel--Cantelli lemma shows that
\begin{equation}
\label{shrinking}
\mu(\mathcal{W}_{\beta}(T,(E_n)))=0 \ \text { if } \ \sum_{n=1}^{\infty}\mu(\{\alpha\in X\colon T^{n}(\alpha)\in \beta+E_n\})<\infty.\end{equation}
Two natural questions then arise$\colon$
\begin{itemize}
\item What is  $\mu\big(\mathcal{W}_{\beta}(T,(E_n))\big)$ when the above series (\ref{shrinking}) diverges?

\item What is the Hausdorff measure/dimension of $\mathcal{W}_{\beta}(T,(E_n))$ when it is  $\mu$-null?
\end{itemize}
Since its introduction, the set $\mathcal{W}_{\beta}(T,(E_n))$ and its variants have been studied in many contexts. These include mixing systems, automorphisms, beta-transformations, Gauss maps, and fractal-restricted targets \cite{AJW25,D25, LLSV25, LLVWZ25, LWWX14,  P67,SW13,WZ21}.

Let $A$ be an integral matrix defining an endomorphism 
\[
T:
\mathbb{R}^{d}/\mathbb{Z}^{d}
\to \mathbb{R}^{d}/\mathbb{Z}^{d},
\qquad
T({\bf x}) = A{\bf x}~\!\!\!\!\pmod1.
\]
In recent work, Li \emph{et al.} \cite{LLVZ23} revisited the shrinking target problem for such toral endomorphisms, originally studied by Hill and Velani \cite{HV99}. For broad class of target sets  $(E_n)_{n\ge1}$ --- including metric balls, rectangular regions, and hyperboloidal neighbourhoods in algebraic varieties --- they determined the Lebesgue measure and Hausdorff dimension of the shrinking target set $\mathcal{W}_{\beta}(T,(E_n)),$ under the so-called summable-mixing property. Their results provide quantitative refinements beyond the classical zero–full law, substantially extending the dichotomy established  in \cite{HV99}. Li \emph{et al.} also allowed the matrix to have real coefficients.

For $i=1,2,\ldots,d$, let $r_i\colon \N \to (0,\frac{1}{2})$, and define ${\bf r} =(r_1,\ldots,r_d).$ For each $n\in \N,$ define the target set
\begin{equation*}
 \mathcal{R}_n({\bf r}):=\{{\bf x}\in [0,1)^{d}\colon |\!|x_i|\!|\le r_i(n),\ 1\le i\le d\}.
\end{equation*}
For ${\bm \beta}=(\beta_1,\ldots,\beta_d)\in [0,1)^{d},$ the shifted set ${\bm \beta}+ \mathcal{R}_{n}({\bf r})$ is a box centred at ${\bm \beta}$. Let $\sigma(A)$ denote the smallest singular value of a matrix $A$, i.e.,
$$\sigma(A):=\inf\big\{|\!|Ax|\!|_{2}\colon |\!|x|\!|_2=1\big\}.$$
For ${\bf x}=(x_1,\ldots,x_d)\in\mathbb{R}^{d},$ write $|{\bf x}|_{\infty}=\max_{1\le i\le d}|x_i|.$
The Fourier transform of a non-atomic probability measure $\mu$ is given by
$$\widehat{\mu}({\bf t})
:=\int e^{-2\pi i\langle{\bf t},{\bf x}\rangle}{\rm d}\mu({\bf x}) \ \ \ \ ({\bf t}\in \mathbb{R}^{d}),$$
where the scalar product $\langle{\bf t},{\bf x}\rangle$ is defined by
$$\langle{\bf t},{\bf x}\rangle=t_1x_1+\cdots+t_dx_d.$$
Recently, Tan and Zhou \cite{TZ25} extended the Lebesgue-measure analysis in \cite{LLVZ23} to general Fourier-decaying measures. Let $\mathcal{A}=(A_n)_{n\ge 1}$ be a sequence of expanding integral matrices --- meaning that the absolute determinant exceeds 1 --- and fix ${\bm \beta}\in[0,1)^{d}.$
They studied the lim sup set
\begin{equation*}
\mathcal{W}_{\bm {\beta}}(\mathcal{A},{\bf r}):= \left\{ {\bm \alpha} \in [0,1)^{d} \colon A_n {\bm {\alpha}} \in {\bm {\beta}}+\mathcal{R}_n({\bf r}) \!\!\!\!\!\pmod{1} \text{ i.m. } n \in \N \right\}.
\end{equation*}

\begin{thm}[{\cite[Theorem 1.10]{TZ25}}]
\label{TZ}
Let $\br: \bN \to (0,1/2)^d$.
Let $\mu$ be a probability measure supported on  $[0,1)^{d}$. Let $\mathcal{A} = (A_n)_{n \ge 1}$ be a sequence of expanding integral matrices such that $\sigma(A_{n+1}A_{n}^{-1}) \ge K>1 ~ (n \ge 0),$
where $A_0$ is the identity matrix in $d$ dimensions. Suppose there exists $s>d+1$ such that
$$\widehat{\mu}({\bf t})=O\Big((\log |{\bf t}|_{\infty})^{-s}\Big) ~~\text{as}~  |{\bf t}|_{\infty} \to \infty.$$
Then \begin{equation*}
\mu(\mathcal{W}_{\bm {\beta}}(\mathcal{A},{\bf r}))=\begin{cases}
   0,   & \text{if } \sum_{n=1}^{\infty}\prod_{i=1}^{d}r_i(n)<\infty, \\
    1,  & \text{if } \sum_{n=1}^{\infty}\prod_{i=1}^{d}r_i(n)=\infty.
\end{cases}
\end{equation*}
\end{thm}

Surprisingly, nothing seems to be known about twisted Diophantine approximation for the matrix-generated dynamical systems. The purpose of this work is to fill this gap.

\subsection{Results}

To state our results, we require some additional notation and terminology. A $d$-tuple of non-increasing functions
$${\bf r} \colon \N\to \left(0,\frac{1}{2}\right)^d$$ is \emph{doubling} if there exists $C>1$ such that for all $1\le i\le d$ and all $n \in \N,$
$$r_{i}(n) \le C r_{i}(2n). $$
Thus, when we write of doubling functions, we implicitly assume their monotonicity.
For ${\bm \alpha}\in[0,1)^{d},$ define 
$$\mathcal{T}_{{\bm \alpha}}(\mathcal{A}, {\bf r})
= \left\{ {\bm \beta} \in [0,1)^{d} \colon A_n {\bm {\alpha}} \in {\bm {\beta}}+\mathcal{R}_n({\bf r}) \!\!\!\!\!\pmod{1} \text{ i.m. } n \in \N \right\}.$$

A \emph{dimension function} is a continuous, non-decreasing function 
\[
f: (0, \infty) \to (0,\infty)
\]
such that $f(\rho) \to 0$ as $\rho\to 0^+.$
For $s\ge 0,$ we write $f\preceq s$ if 
$$\frac{f(y)}{y^{s}}\le \frac{f(x)}{x^{s}} \text { for all }0<x<y.$$
Observe that if $f \preceq s$ then
$$\lim_{\rho\to 0^{+}}\frac{f(\rho)}{\rho^{s}}>0.$$
%If this limit is infinite, then we write $f\prec s.$ The notations $s\preceq f$ and $s\prec f$ are defined similarly.

We now state our first result: a Khintchine-type dichotomy for twisted approximation for general Fourier-decaying measures.

\begin{thm}\label{t1}
Let ${\bf r}\colon \N\to (0,\frac{1}{2})^d$ be doubling, and let $\mu$ be a probability measure on $[0,1)^{d}$.
Let $(A_n)_{n\in\mathbb{N}}$ be a sequence of integral $d \times d$ matrices,
and let $A_0$ be the identity matrix in $d$ dimensions. 
Assume that, for some  $c, s, \eps>0,$ we have
\begin{equation}
\label{sing}
\sigma(A_n-A_m)\ge c|n-m|^{\frac{1}{s}}(\log_+ |n-m|)^{\frac{1}{s}+\eps}
\end{equation}
whenever $m\ne n$. Assume also that 
$$\widehat{\mu}({\bf {t}})=O(|{\bf t}|_{\infty}^{-s}) ~~\text{as}~  |{\bf t}|_{\infty} \to \infty.$$
Then, for $\mu$-almost all ${\bm \alpha}$,
\begin{equation*}
\mathcal{L}^d\big(\mathcal{T}_{{\bm \alpha}}(\mathcal{A}, {\bf r})\big)=\begin{cases}
   0,   & \text{if  $\sum_{n=1}^{\infty}\prod_{i=1}^{d}r_i(n)<\infty$}, \\
    1,  & \text{if  $\sum_{n=1}^{\infty}\prod_{i=1}^{d}r_i(n)=\infty$}.
\end{cases}
\end{equation*}
\end{thm}
\iffalse
\begin{rem}
The condition \eqref{sing} is weaker than the condition \eqref{sing0}. Indeed, if we assume \eqref{sing0} then, for $m < n$,
we have
\begin{align*}
|\!|(A_n-A_m)^{T}({\bf k})|\!|_{2}&=|\!|A_m^{T}(A_nA_m^{-1}-I)^{T}({\bf k})|\!|_{2}
\\&\ge \sigma(A_m)|\!|(A_nA_m^{-1}-I)^{T}({\bf k})|\!|_{2}\\
&\ge  \sigma(A_m)\sigma(A_nA_m^{-1}-I)|\!|{\bf k}|\!|_{2},
\end{align*}
where $I$ is the identity matrix. By induction, we have $\sig(A_m) \ge K^m$ and
\begin{equation*}
\label{calc}
\sig(A_n A_m^{-1} - I)
\ge \sig(A_n A_m^{-1}) - 1 \ge K^{n-m} - 1,
\end{equation*}
whence
\[
|\!|(A_n-A_m)^{T}({\bf k})|\!|_{2} \ge K^{m}(K^{n-m}-1)|\!|{\bf k}|\!|_{2}.
\]
For $K > 1$, there exists $m_0 \in \mathbb{N}$ such that if $n > m \geq m_0$, 
$$K^{m} \geq 1 + \frac{1}{K^{n-m}-1} \geq \frac{K^{n-m}}{K^{n-m} - 1},$$ which implies $\sigma(A_n - A_m) \geq K^{n-m}$.
   
\end{rem}
\fi

\begin{example} Let $\mu$ be the induced Lebesgue measure on an open ball intersected with a smooth submanifold of $\bR^d$ that is not contained in any affine hyperplane. Then $\mu$ has polynomial Fourier decay rate \cite[\S VIII.3.2]{Ste1993}, and the theorem applies for some $s$.
\end{example}

An application of Fubini's theorem delivers the following observation.
\begin{pro}
\label{p0}
Let ${\bf r}\colon \N\to (0,\frac{1}{2})^d$, and let $\mu\otimes\nu$ be a probability measure on $[0,1)^{d}\times [0,1)^{d}.$ Let $\mathcal{A}=(A_n)_{n\in\mathbb{N}}$ be a sequence of integral $d \times d$ matrices, and let $A_0$ be the identity matrix in $d$ dimensions. Assume that the set 
$$\left\{ ({\bm \alpha},{\bm \beta}) \in [0,1)^{d}\times[0,1)^{d} \colon A_n {\bm {\alpha}} \in {\bm {\beta}}+\mathcal{R}_n({\bf r}) \!\!\!\!\!\pmod{1} \text{ i.m. } n \in \N \right\}$$
has full $\mu\otimes\nu$-measure. Then
for $\mu$-almost all ${\bm \alpha}\in [0,1)^{d}$,  
\begin{equation*}
\nu\big(\mathcal{T}_{{\bm \alpha}}(\mathcal{A}, {\bf r})\big)=\begin{cases}
   0,   & \text{if  $\sum_{n=1}^{\infty}\prod_{i=1}^{d}r_i(n)<\infty$}, \\
    1,  & \text{if  $\sum_{n=1}^{\infty}\prod_{i=1}^{d}r_i(n)=\infty$}.
\end{cases}
\end{equation*}
\end{pro}

The result below follows immediately from Theorem \ref{TZ} and Proposition $\ref{p0}$.

\begin{cor}
\label{t0}
Let ${\bf r}\colon \N\to (0,\frac{1}{2})^d$, and let $\mu$ be a probability measure on $[0,1)^{d}$.
Let $(A_n)_{n\in\mathbb{N}}$ be a sequence of integral $d \times d$ matrices, and let $A_0$ be the identity matrix in $d$ dimensions. 
Assume that, for some $K > 1$, we have
\begin{equation}\label{sing0}
\sigma(A_{n+1}A_{n}^{-1}) \ge K \qquad (n \ge 0).
\end{equation}
Assume also that there exists $s>d+1$ such that
$$\widehat{\mu}({\bf {t}})=O\Big((\log |{\bf t}|_{\infty})^{-s}\Big) ~~\text{as}~  |{\bf t}|_{\infty} \to \infty.$$
Then, for $\mu$-almost all ${\bm \alpha}$,
\begin{equation*}
\mathcal{L}^d\big(\mathcal{T}_{{\bm \alpha}}(\mathcal{A}, {\bf r})\big)=\begin{cases}
   0,   & \text{if  $\sum_{n=1}^{\infty}\prod_{i=1}^{d}r_i(n)<\infty$}, \\
    1,  & \text{if  $\sum_{n=1}^{\infty}\prod_{i=1}^{d}r_i(n)=\infty$}.
\end{cases}
\end{equation*}
\end{cor}

In Theorem \ref{t1} and Corollary \ref{t0}, if $\mu$ is Lebesgue measure, then the condition (\ref{sing0}) can be relaxed. Moreover, in this case we do not require $\br$ to be doubling.

\begin{thm}\label{t2}
Let ${\bf r}\colon \N\to (0,\frac{1}{2})^d$. Let $(A_n)_{n\in\mathbb{N}}$ be a sequence of integral $d \times d$ matrices, and let $A_0$ be the identity matrix in $d$ dimensions. Assume that $\det(A_m - A_n) \ne 0$ whenever $m \ne n$.
Then, for $\mathcal{L}^{d}$-almost all ${\bm \alpha}\in [0,1)^{d}$,  
\begin{equation*}
\mathcal{L}^d\big(\mathcal{T}_{{\bm \alpha}}(\mathcal{A}, {\bf r})\big)=\begin{cases}
   0,   & \text{if  $\sum_{n=1}^{\infty}\prod_{i=1}^{d}r_i(n)<\infty$}, \\
    1,  & \text{if  $\sum_{n=1}^{\infty}\prod_{i=1}^{d}r_i(n)=\infty$}.
\end{cases}
\end{equation*}
\end{thm}
\begin{rem}
The determinant condition is weaker than either of the conditions \eqref{sing}, \eqref{sing0}. Indeed, \eqref{sing} implies the determinant condition trivially, since it gives
$\sig(A_n - A_m)>0$ whenever $m \ne n$. Moreover, if we instead assume \eqref{sing0}, then for $m < n$ we have
$$\sigma(A_n-A_m)\ge \sigma(A_m)\sigma(A_nA_m^{-1}-I)\ge \sigma(A_m)(K^{n-m}-1)>0,$$
where $I$ denotes the identity matrix.

\end{rem}

Let ${\bm \omega}=({\bm \omega_{n}})_{n\ge 1}$ be sequence of points in $\mathbb{R}^{d},$ wherein we abbreviate ${\bm \omega_{n}}=(\omega_{n,1},\ldots,\omega_{n,d})\in \mathbb{R}^{d}.$ Define the twisted lim sup set 
$$\mathcal{T}({\bm \omega}, {\bf r}):=\left\{ {\bm \beta} \in [0,1)^{d} \colon {\bm \omega_{n}} \in {\bm {\beta}}+\mathcal{R}_n({\bf r}) \!\!\!\!\!\pmod{1} \text{ i.m. } n \in \N \right\}.$$
Gonz\'{a}lez Robert \emph{et al.} \cite{RHSW25} proposed the following conjecture.

\begin{conj}[{\cite[Conjecture 1.10]{RHSW25}}]\label{conj}
For almost every sequence ${\bm \omega}$ in $\mathbb{R}^{d},$ for every $d$-tuple of non-increasing functions ${\bf r}=(r_1,\ldots,r_d),$ we have
\begin{equation*}
\mathcal{L}^d\big(\mathcal{T}({\bm \omega}, {\bf r})\big)=\begin{cases}
   0,   & \text{if  $\sum_{n=1}^{\infty}\prod_{i=1}^{d}r_i(n)<\infty$}, \\
    1,  & \text{if  $\sum_{n=1}^{\infty}\prod_{i=1}^{d}r_i(n)=\infty$}.
\end{cases}
\end{equation*}
\end{conj}

Let $(A_n)_{n\ge 1}$ be a sequence of integral matrices satisfying the assumptions of Theorem \ref{t2}, and let ${\bm \omega}=(A_n{\bm \alpha})_{n\ge1}$. We see that
Theorem \ref{t2} solves, affirmatively, a variant of Conjecture \ref{conj}. Moreover, the result remains valid for general ${\bf r}$, relaxing monotonicity.

Returning to the convergence case of Theorem \ref{t2}, the Lebesgue measure of $\mathcal{T}_{{\bm \alpha}}(\mathcal{A}, {\bf r})$ is zero, offering no further insight into its size. Intuitively, the size of $\mathcal{T}_{{\bm \alpha}}(\mathcal{A}, {\bf r})$ should decrease as the rate of approximation governed by the function ${\bf r}$ increases.
In short, we require a more delicate notion of size than simply Lebesgue measure. Generalised Hausdorff measures serve as suitable framework for capturing the fine metric structure of  $\mathcal{T}_{{\bm \alpha}}(\mathcal{A}, {\bf r})$.
The Hausdorff $f$-measure $\mathcal{H}^{f},$ defined via a dimension function $f,$ naturally extends the Lebesgue measure. We refer the reader to \S \ref{Hausdorff} for standard definitions of $\mathcal{H}^{f}$ and Hausdorff dimension.

In view of the development of the Lebesgue theory, it is natural to ask for conditions under which $\mathcal{H}^{f}\big(\mathcal{T}_{{\bm \alpha}}(\mathcal{A}, {\bf r})\big)$ can be determined. The following result provides a simple criterion.

\begin{thm}\label{t3}
Let ${\bf r}\colon \N \to (0,\frac{1}{2})^d$ with $|\br(n)|_\infty \to 0$ as $n \to \infty$. Let $(A_n)_{n\in\mathbb{N}}$ be a sequence of integral $d \times d$ matrices, and let $A_0$ be the identity matrix in $d$ dimensions. Assume that $\det(A_m - A_n) \ne 0$ whenever $m \ne n$. Let $f$ be a dimension function such that $s \preceq f \preceq s+1$ for some integer $s \in [0,d-1]$.
Then, for $\mathcal{L}^{d}$-almost all ${\bm \alpha}\in [0,1)^{d}$, we have
\begin{equation*}
\mathcal{H}^{f}\big(\mathcal{T}_{{\bm \alpha}}(\mathcal{A}, {\bf r})\big)=\begin{cases}
   0,   & \text{ if  } \sum_{n=1}^{\infty}s_{n}({\bf r},f)<\infty, \\
    \mathcal{H}^{f}([0,1)^{d}),  & \text{ if  }\sum_{n=1}^{\infty}s_{n}({\bf r},f)=\infty,
\end{cases}
\end{equation*}
where 
$$s_{n}({\bf r},f) = \min_{1\le i\le d}\left\{f(r_i(n))\prod_{j\colon r_{j}(n)
 \ge r_{i}(n)}\frac{r_j(n)}{r_i(n)}\right\}.$$
\end{thm}

Applying Theorem \ref{t3}, we immediately obtain the following corollary.

\begin{cor}
\label{c1}
Let ${\bf r}\colon \N\to (0,\frac{1}{2})^d$ with
$$r_{i}(n)=n^{-\tau_i} ~~(1\le i\le d) ~\text{ and } ~\sum_{i=1}^{d}\tau_{i}> 1.$$
Then for Lebesgue almost all ${\bm \alpha}\in [0,1)^{d}$, 
$$\dim_{\rm H}\mathcal{T}_{{\bm \alpha}}(\mathcal{A}, {\bf r})=\min_{1\le i\le d}\left\{\frac{1+\sum_{j\colon \tau_{j}<\tau_i}(\tau_i-\tau_j)}{\tau_i}\right\}.$$
\end{cor}

\begin{rem}
Let us compare this to Theorem \ref{B-th}. Taking 
\[
\tau_1=\cdots=\tau_{d}=\tau>\frac{1}{d}
\]
in Corollary \ref{c1} shows that the Hausdorff dimension of  $\mathcal{T}_{{\bm \alpha}}(\mathcal{A}, {\bf r})$ is also $\frac{1}{\tau}$, irrespective of the ambient dimension. Some may find this outcome to be counter-intuitive.
\end{rem}

\subsection{Open problems}

We expect many of our results, such as Theorem \ref{t1}, to hold subject to a Diophantine condition on $\balp$. Such results might refine the ones that we have presented herein.

\subsection*{Notation}
Throughout, we use the standard Bachmann--Landau notations. For functions $f, g\colon X\to \mathbb{R}$, we write $f\ll g$ to if there exists a constant $C > 0$ such that $|f(x)|\le C |g(x)|$ for all $x\in X$. We write $f\asymp g$ if $f\ll g \ll f$. 
We also use the abbreviations $e(x)=\text{exp}{(2\pi i x)}$ and $\log_+ m = \max \{ 1, \log m \}$.

\section{Preliminaries}
\subsection{Hausdorff measure and content}
\label{Hausdorff}
Let $f$ be a dimension function. For $E\subset \R^d$ and $\eta>0$, define
\[\mathcal H_\eta^f(E)=\inf\bigg\{\sum_{i}f(|B_i|):E\subset \bigcup_{i\ge 1}B_i, \quad |B_i|\le \eta \bigg\},\]
where $|B|$ denotes the diameter of a ball $B$. The definition also applies when $\eta=\infty$, and the quantity $\mathcal{H}_\infty^f(E)$ is called the {\em Hausdorff $f$-content} of $E$. 
The {\em Hausdorff $f$-measure} of $E$ is 
\[
\mathcal{H}^f(E):=\lim_{\eta\to 0^+}\mathcal H_\eta^f(E).\]

\begin{lemma}[{\cite[Lemma 1.2.5]{BP17}}]
Let $ f$ and $g$ be dimension functions such that the ratio $f(r)/g(r) \to 0 $ as $ r \to 0 $. 
If $\mathcal{ H}^{g} (E) < \infty $, 
then $\mathcal{ H}^{f} (E) =0.$
\end{lemma}

When  $f(r) = r^s$ with $s \ge 0$, the measure $ \mathcal{H}^{f}$ coincides with the standard
{\em $s$--dimensional Hausdorff measure}
$\mathcal{H}^{s}.$ When $s \in \N$, the measure
$\mathcal{H}^{s}$ is proportional to 
$s$-dimensional Lebesgue measure. The \emph{Hausdorff dimension} of a set $E$ is 
$$
\dim_{\rm H} E 
:=  \inf \left\{ s\colon \mathcal{H}^{s} (E) =0 \right\}
=\sup \left\{ s\colon \mathcal{ H}^{s} (E) = \infty \right\}. 
$$ 
For more background, see \cite{det,falc}. 

A classical and widely-used method for establishing a lower bound on the Hausdorff $f$-measure of a set $E$ is the mass distribution principle.
The following version is a slight generalisation of {\cite[Lemma 1.2.8]{BP17}} and follows from the same argument.

\begin{lemma}
\label{p:MDP}
Let $ E $ be a Borel subset of $ \R^d $. If $ E $ supports a Borel probability measure $ \mu $ that satisfies
\[
\mu(B)\le cf(|B|),
\]
for some constant $ 0<c<\infty $ and for every ball $B$, then 
$\mathcal{H}_{\infty}^f(E)\ge 1/c$.
\end{lemma}

A key tool in proving the divergence part of Theorem \ref{t3}  is the following `balls-to-open sets' mass transference principle. Originally established by Koivusalo and Rams \cite{KR21} in the context of Hausdorff dimension, this principle was later extended to Hausdorff measures by Zhong~\cite{Z21} and by He \cite{H25}, with conditions on the singular value function and on the Hausdorff content, respectively.

\begin{lemma}
[{\cite[Theorems 2.4 and 2.5]{H25}}]
\label{t:weaken}
Let $f$ be a dimension function such that $f\preceq d$. Let $(B_k)_{k=1}^\infty$ be a sequence of balls in $[0,1]^d$ with radii tending to 0 such that $\mathcal{L}^d(\limsup B_k)=1$. Let $(E_k)_{k=1}^\infty$ be a sequence of open sets such that $E_k \subseteq B_k$ for all $k$. If there exists a constant $c>0$ such that for any $k\ge 1$,
\[
\mathcal{H}_{\infty}^f (E_k)>c\mathcal{L}^d(B_k),
\]
then
\[
\mathcal{H}^f \Big(\limsup_{k\to\infty} E_k\Big)= \mathcal{H}^f([0,1]^d).\]
\end{lemma}

Equipped with these results, the main strategy for proving the divergence part of Theorem \ref{t3} is to verify that the relevant sets satisfy suitable lower bounds on their Hausdorff $f$-content. Lemma \ref{p:MDP} will play a crucial role.

\subsection{Auxiliary results}

The following lemmas are widely used to estimate the measure of a lim sup set in a probability space.

\begin{lemma}[Borel--Cantelli lemma] \label{BC lemma} Let $(\Omega, \mathcal{B}, \nu)$ be a probability space. Let 
$(E_n)_{n=1}^\infty$ be a sequence of measurable sets, and let
\[
E=\limsup E_n=\bigcap_{N=1}^{\infty}\bigcup_{n=N}^{\infty}E_n.
\]
If $\sum_{n= 1}^{\infty}\nu(E_n)<\infty$ then $\nu(E) = 0$. If instead $\sum_{n= 1}^{\infty}\nu(E_n) = \infty$, and the sets $E_n$ are pairwise independent, then
$\nu(E) = 1$.
\end{lemma}

We note that the divergence part of the above lemma requires the independence of the sets $E_n$. In applications where pairwise independence fails, the following lemma is often used as a proxy.

\begin{lemma}[Chung--Erd\H{o}s inequality \cite{Chung}]\label{CE}
If $\sum_{n\ge 1}\nu(E_n)=\infty$, then $$
\nu(E)\ge \limsup_{N\to \infty}\frac{(\sum_{1\le n\le N}\nu(E_n))^2}{\sum_{1\le m,n\le N}\nu(E_m\cap E_n)}.
$$
\end{lemma}

This is also known as the divergence Borel--Cantelli lemma \cite{BV2023}. In many cases, the Chung--Erd\H{o}s inequality only tells us that $\nu(E)>0$. To obtain a full measure result for $E$, one may apply the inequality locally and then deduce that $\nu(E)=1$
via the following lemma.

\begin{lemma}[{\cite[Lemma 6]{BDV06}}]
\label{Full meas}
Let $(\Omega, d)$ be a metric space with a finite measure $\nu$ for which every open set is measurable. Suppose $E \subseteq \Omega$ is a Borel set and $h: [0,\infty) \to [0,\infty)$ is an increasing function such that $h(x) \to 0$ as $x \to 0^+$. If the inequality
$$\nu(E\cap U)\ge h(\nu(U)),$$
holds for every open set $U \subset \Omega$, then $\nu(E) = \nu(\Omega)$. 
\end{lemma}

The next lemma establishes a fundamental connection between the Fourier decay of a measure and the equidistribution properties of sequences of expanding integral matrices.

\begin{lemma}
\label{equi}
Let $\mu$ be a probability measure on $\mathbb{R}^{d}$ with polylogarithmic Fourier decay. Let $(A_n)_{n\in\N}$ be a sequence of expanding integral matrices. Assume that the inequality \eqref{sing} holds. Then the orbit $(A_n{\bm \alpha})_{n\ge 1}$ is equidistributed modulo $\bZ^d$ for $\mu$-almost every ${\bm \alpha}.$
\end{lemma}

\begin{proof}
A proof follows along similar lines to that of {\cite[Theorem 1.7]{TZ25}}.
\end{proof}

The following calculation will play a pivotal role in the estimation procedure.

\begin{lemma}
\label{elementary}
For 
$s, \eps >0$,
\[
S := \sum_{m=1}^{\infty}\sum_{{\bf k} \in \mathbb{Z}^{d} \setminus \{{\bf 0}\}} (m^\frac{1}{s}(\log_+ m)^{\frac{1}{s}+\eps}|{\bf k}|_{\infty})^{-s} \prod_{i=1}^{d}
\min \{ 1, |k_i|^{-1} \}  
< \infty.
\]
\end{lemma}

\begin{proof}
We compute that
\begin{align*}
S &\le \sum_{m=1}^{\infty}
\sum_{j=1}^{d}\sum_{\substack{{\bf k}\in \mathbb{Z}^{d} \setminus \{{\bf 0}\} \\ |{\bf k}|_{\infty}=|k_j|}} 
(m^\frac{1}{s}(\log_+ m)^{\frac{1}{s}+\eps}|{\bf k}|_{\infty})^{-s}
\prod_{i=1}^{d}\min \{1, |k_i|^{-1}\} \\
&\le 2d \sum_{m=1}^\infty
\sum_{k_1=1}^\infty (m^\frac{1}{s}(\log_+ m)^{\frac{1}{s}+\eps}k_1)^{-s} k_1^{-1}
\left(
\sum_{|k| \le k_1} \min \{ 1, |k|^{-1} \} \right)^{d-1}.
\end{align*}
Therefore
\[
S \ll_d \sum_{m=1}^\infty m^{-1}(\log_+m)^{-1-s\eps} \sum_{k_1=1}^\infty
k_1^{-1-s}(\log k_1)^{d-1}< \infty.
\]
\iffalse
We partition $\N$ into $d+1$ classes, setting $\delta_{d+1}=1$ for convenience. Let $\Omega_1=\{n\in\N\colon n\ge \frac{1}{\delta_1}\}.$ Then
\begin{align*}
\sum_{|k_1|\in\Omega_1}|k_1|^{-1}(\log(K^{m-1}|{k}_{1}|))^{-s}&\cdot \Big(\delta_2+\sum_{1\le |k_2|<\frac{1}{\delta_2}}\delta_2+\sum_{\frac{1}{\delta_2}\le |k_2|\le|k_1|}|k_2|^{-1}\Big) \\ \cdots&\Big(\delta_d+\sum_{1\le |k_d|<\frac{1}{\delta_d}}\delta_d+\sum_{\frac{1}{\delta_d}\le |k_d|\le|k_1|}|k_d|^{-1}\Big)\\
\ll (\log(K^{m-1}))^{-(s-d)}.
\end{align*}
Let $1 \le i \le d$ and $\Omega_{i+1}=\{n\in\N\colon \frac{1}{\delta_{i+1}}\le n\le \frac{1}{\delta_{i}}\}$. Then
\begin{align*}
\sum_{|k_1|\in\Omega_{i+1}}&\delta_1(\log(K^{m-1}|{k}_{1}|))^{-s}\cdot \Big(\delta_2+\sum_{1\le |k_2|<|k_1|}\delta_2\Big)\cdots  \Big(\delta_i+\sum_{1\le |k_i|\le |k_1|}\delta_i\Big)  \\&\Big(\delta_{i+1}+\sum_{1\le |k_{i+1}|\le \frac{1}{\delta_{i+1}}}\delta_{i+1}+\sum_{ \frac{1}{\delta_{i+1}}\le |k_{i+1}|\le |k_1|}|k_{i+1}|^{-1}\Big)\cdots\Big(\delta_{d}+\sum_{1\le |k_{d}|\le \frac{1}{\delta_{d}}}\delta_{d}+\sum_{ \frac{1}{\delta_{d}}\le |k_{d}|\le |k_1|}|k_{d}|^{-1}\Big)
\\
&\ll (\log(K^{m-1}))^{-(s-(d-i))}.
\end{align*}
Since $s>d+1,$ combining our estimates gives
 $$S\ll \sum_{m=1}^{\infty}(\log(K^{m-1}))^{-(s-d)}<\infty.$$
\fi
\end{proof}

\section{Proof of Theorem \ref{t1}}

Recall that 
\begin{align*}
\mathcal{T}_{{\bm \alpha}}(\mathcal{A}, {\bf r})&=\{{\bm \beta}\in[0,1)^{d}\colon A_n{\bm \alpha} \in {\bm \beta}+\mathcal{R}_n({\bf r})\!\!\!\!\pmod1 \text{ i.m. }  n\in\mathbb{N} \}\\
&= \limsup E_n,
\end{align*}
where 
$$E_n=\{{\bm \beta}\in[0,1)^{d}\colon A_n{\bm \alpha} \in {\bm \beta}+\mathcal{R}_n({\bf r})\!\!\!\!\pmod1\}.$$

\subsection{Convergence part}

We assume that 
$$\sum_{n=1}^{\infty}\prod_{i=1}^{d}r_i(n)<\infty.$$
It is readily checked that 
$$\mathcal{L}^{d}(E_n)=2^{d}\prod_{i=1}^{d}r_i(n).$$
Now Lemma \ref{BC lemma} yields 
\[
\cL^d(\mathcal{T}_{{\bm \alpha}}(\mathcal{A}, {\bf r})) = 0
\]
for any ${\bm \alpha}\in[0,1)^{d}.$

\subsection{Divergence part}

In this case, we assume that the series
\[
\sum_{n=1}^{\infty}\prod_{i=1}^{d}r_i(n) = \infty.
\]
We proceed in two stages:
\begin{enumerate}
\item Construct a suitable set $G\subseteq[0,1)^{d}$ such that $\mu(G) = 1$.
\item Show that  $\cL^d(\mathcal{T}_{{\bm \alpha}}(\mathcal{A}, {\bf r})) = 1$ for every ${\bm \alpha}\in G.$
\end{enumerate}

\bigskip

For $N, p\in\mathbb{N},$  define
\begin{equation*}
B_{N}(p) =
\left\{
\begin{split}
{\bm \alpha}\in[0,1)^{d}:
&\sum_{1\le m,n\le N} \mathcal{L}^{d}(E_m\cap E_n) \\ &> p\int\sum_{1\le m,n\le N}\mathcal{L}^{d}(E_m\cap E_n){\rm d}\mu 
\end{split}
\right\}.
\end{equation*}
Chebyshev's inequality yields $\mu(B_N(p)) \le 1/p$, whence
$$\mu(B_{N}^{c}(p))\ge 1-\frac{1}{p},$$
where $B_{N}^{c}(p)$ is the complement of $B_{N}(p).$   Defining
$$B(p)=\limsup_{N\to\infty}B_{N}^{c}(p), \qquad B=\bigcup_{p=1}^{\infty}B(p),$$
we have $\mu(B) = 1$ by the reverse Fatou lemma.

\begin{lemma}
\label{qusi}
In the context of Theorem \ref{t1}, for $N$ sufficiently large, we have
$$\int\sum_{1\le m,n\le N}\mathcal{L}^{d}(E_m\cap E_n){\rm d}\mu({\bm \alpha})\le C_1 \left(\sum_{1\le n\le N}\mathcal{L}^{d}(E_n)\right)^{2},$$
where $C_1 = C_1(d,c,s,\eps)$.
\end{lemma}

\begin{proof}
By symmetry, it suffices to bound the contribution from $m<n.$ Observe that if $E_m\cap E_n\neq \emptyset$ then
$$(A_n-A_m){\bm \alpha} \bmod 1 \in\mathcal{R}_m(2{\bf r}).$$
It follows that 
$$\mathcal{L}^{d}(E_m\cap E_n)\le \mathcal{L}^{d}(E_n)\cdot {\bm 1}_{\mathcal{R}_m(2{\bf r})}((A_n-A_m){\bm \alpha}),$$
where ${\bm 1}_{\mathcal{R}_m(2{\bf r})}({\bf x})=\prod_{i=1}^{d}{\bm 1}_{([-2r_i,2r_i])}(x_i)$ is the characteristic function of the rectangular domain $\mathcal{R}_m( 2{\bf r}).$ Now
\begin{align}
\notag
&\int\sum_{1\le m<n\le N}\mathcal{L}^{d}(E_m\cap E_n){\rm d}\mu({\bm \alpha})
\\
\notag
&\le \sum_{1\le m<n\le N}\mathcal{L}^{d}(E_n)\int{\bm 1}_{\mathcal{R}_m(2{\bf r})}((A_n-A_m){\bm \alpha}) {\rm d}\mu({\bm \alpha})\\
\label{Qusi}
&=\sum_{1\le m<n\le N}\mathcal{L}^{d}(E_n)\sum_{{\bf k}\in \mathbb{Z}^{d}} \hat {\bm 1}_{\mathcal{R}_m(2{\bf r})}({-\bf k}) \hat{\mu}((A_n-A_m)^{T}{\bf k}).
\end{align}

The Fourier coefficients
\[
\hat {\bm 1}_{\mathcal{R}_m(2{\bf r})}({\bf k}) = \int{\bm 1}_{\mathcal{R}_m(2{\bf r})}({\bf x})e(-\langle{\bf k},{\bf x}\rangle){\rm d} {\bf x}
\]
satisfy
$$\hat {\bm 1}_{\mathcal{R}_m(2{\bf r})} (\bzero) = 2^{d}\mathcal{L}^{d}(E_m) $$
and
$$|\hat {\bm 1}_{\mathcal{R}_m(2{\bf r})}(\bk)|\le \prod_{i=1}^{d}\min\left\{4r_i(m),\frac{1}{|k_i|}\right\} \text{ for } {\bf k}\neq {\bf 0}.$$

Combining (\ref{sing}) with the Fourier decay assumption gives
\begin{align*}
&\sum_{{\bf k}\in \mathbb{Z}^{d}} \hat {\bm 1}_{\mathcal{R}_m(2{\bf r})}({\bf k}) \hat{\mu}((A_n-A_m)^{T}{\bf k}) \\
&\ll_c
\mathcal{L}^{d}(E_m) + \sum_{{\bf k}\in \mathbb{Z}^{d} \setminus \{{\bf 0}\}} \frac{\prod_{i=1}^{d} \min \left \{4r_i(m),\frac{1}{|k_i|}\right\}}{(|n-m|^{\frac{1}{s}}(\log_+ |n-m|)^{\frac{1}{s}+\eps}|{\bf k}|_{\infty})^{s}}.
\end{align*}
Inserting this into \eqref{Qusi} furnishes
\[
\int\sum_{1\le m<n\le N}\mathcal{L}^{d}(E_m\cap E_n){\rm d}\mu({\bm \alpha})
\ll_{c, d, s,\eps} \left(\sum_{n \le N}\mathcal{L}^{d}(E_n)\right)^{2}
\]
for $N$ sufficiently large, by Lemma \ref{elementary} and the divergence of the series $\sum_{n=1}^{\infty}\mathcal{L}^{d}(E_n).$  
\end{proof}

We see from Lemma \ref{qusi} that if $N$ is sufficiently large then $B_N^{c}(p)$ is a subset of 
\begin{align*}
&\widetilde{B}_{N}^{c}(C_1,p):= \\
&\left\{{\bm \alpha}\in[0,1)^{d}\colon \sum_{1\le m,n\le N}\mathcal{L}^{d}(E_m\cap E_n)\le pC_1\left(\sum_{1\le n\le N}\mathcal{L}^{d}(E_n)\right)^{2}\right\}.
\end{align*}
Thus, the set 
$$\widetilde{B}(C_1):=\bigcup_{p=1}^{\infty}\limsup_{N\to\infty}\widetilde{B}_{N}^{c}(C_1,p)$$ has full $\mu$-measure. Further, Lemma \ref{equi} provides a set $D \subseteq [0,1)^d$ with $\mu(D) = 1$ such that $(A_n \bm{\alpha})_{n\ge1}$ is equidistributed for all $\bm{\alpha} \in D$. We  define our target set by
$$G=\widetilde{B}(C_1)\cap D.$$
Then $\mu(G)=1$, which completes the first stage of the proof.

\bigskip

Let ${\bm \alpha}\in G,$ and assume that $$\sum_{n=1}^{\infty}\prod_{i=1}^{d}r_i(n)=\infty.$$
To complete the proof of Theorem \ref{t1}, it remains to show that
\begin{equation}
\label{toshow}
\mathcal{L}^{d}(\mathcal{T}_{{\bm \alpha}}(\mathcal{A}, {\bf r}))=1.
\end{equation}

Since ${\bm \alpha} \in G$, there exist an integer $p \ge 1$ and a sequence $(N_k)_{k=1}^\infty$ satisfying $N_k \to \infty$ such that
$$\sum_{1\le m,n\le N_k}\mathcal{L}^{d}(E_n\cap E_m)\le pC_1\left(\sum_{1\le n\le N_k}\mathcal{L}^{d}(E_n)\right)^{2}$$
holds for all $k\in\N.$  For an arbitrary open set $U\subseteq[0,1)^{d},$ by Lemma~\ref{CE}, we have
\begin{align*}
\mathcal{L}^{d}\Big(\limsup_{n\to\infty}(E_n\cap U)\Big)
&\ge \limsup_{N\to\infty}\frac{\big(\sum_{n=1}^{N}\mathcal{L}^{d}(E_n\cap U)\big)^{2}}{\sum_{m,n=1}^{N}\mathcal{L}^{d}(E_m\cap E_n\cap U)}
\\ &\ge \limsup_{k\to\infty}\frac{\big(\sum_{n=1}^{N_k}\mathcal{L}^{d}(E_n\cap U)\big)^{2}}{pC_1\big(\sum_{n=1}^{N_k}\mathcal{L}^{d}(E_n)\big)^{2}}.
\end{align*}

We claim that, for some constant $C = C_d$, if $k$ is sufficiently large then
$$\sum_{1\le n\le N_k}\mathcal{L}^{d}(E_n)\le \frac{C}{\mathcal{L}^{d}(U)}\sum_{1\le n\le N_k}\mathcal{L}^{d}(E_n\cap U).$$
Assuming the claim for now, we can continue our calculation to obtain
$$\mathcal{L}^{d}\Big(\limsup_{n\to\infty} E_n\cap U \Big)
\ge \frac{(\mathcal{L}^{d}(U))^{2}}{pC_1C^{2}}.$$
Applying Lemma \ref{Full meas} with
$h(x)=\frac{x^{2}}{pC_1C^{2}}$ delivers \eqref{toshow} and completes the proof of Theorem \ref{t1} (subject to the claim).

\bigskip

Finally, we turn to the proof of the claim. Since the open set $U$ can be partitioned into a countable union of axis-aligned cubes $D_1, D_2, \ldots$ of non-increasing volume,
there exists $L$ such that 
$$\sum_{\ell=1}^{L}
\mathcal{L}^{d}(D_\ell)< \mathcal{L}^{d}(U) < 2\sum_{\ell=1}^{L}
\mathcal{L}^{d}(D_\ell).$$
We decompose $\{ 1,2,\ldots,N_k \}$ into discrete subintervals $I_t$, where the minimum $n_t$ and the length of $I_t$ each have order of magnitude $2^t$.
For $k$ sufficiently large, 
\begin{align*}
&\sum_{n \le N_k} \cL^{d}(E_n\cap U) \\
&\ge \sum_{\ell \le L} \sum_t \sum_{n \in I_t} \cL^{d}
(E_n\cap D_\ell) \\
&\gg_d \sum_{\ell \le L}  \sum_t \cL^{d}(E_{n_t}) 
\cdot \sharp
\{ n \in I_t: A_n{\bm \alpha} \in D_\ell \!\!\!\!\pmod1\}.
\end{align*}
As $(A_n \balp)_{n=1}^\infty$ is equidistributed modulo $\bZ^d$, and $\br$ is doubling, we now have
\begin{align*}
\sum_{n \le N_k} \cL^d(E_n \cap U)
&\gg \sum_{\ell \le L} \sum_t
\cL^d(E_{n_t})
2^{t} \cL^d(D_\ell)
\\
&\gg \sum_{\ell \le L} \sum_t \sum_{n \in I_t} \cL^d(E_n) \cL^d(D_\ell)\\
&\gg \cL^d(U) \sum_{n \le N_k} \cL^d(E_n).
\end{align*}
This confirms the claim and completes the proof of Theorem \ref{t1}.

\section{Proof of Theorem \ref{t2}}

\begin{lemma}
\label{independent}
The events 
$$F_n=\{({\bm \alpha}, {\bm \beta})\in[0,1)^{d}\times [0,1)^{d}\colon A_n{\bm \alpha} \in {\bm \beta}+\mathcal{R}_n({\bf r})\!\!\!\!\pmod1\}
\quad (n \in \bN)
$$
are pairwise independent with respect to $\cL^{2d} = \mathcal{L}^{d} \otimes\mathcal{L}^{d}.$
\end{lemma}
\begin{proof}
Recall that ${\bm 1}_{\mathcal{R}_n({\bf r})}({\bf x})=\prod_{i=1}^{d}{\bm 1}_{([-r_i,r_i])}(x_i)$ is the characteristic function defined on the rectangular domain $\mathcal{R}_n({\bf r})$
 for $n\in\N.$ For $m, n\in\N$ with $m\neq n,$ $$\cL^{2d}(F_m\cap F_n)
=\iint{\bm 1}_{\mathcal{R}_m({\bf r})}(A_m{\bm \alpha}-{\bm \beta}){\bm 1}_{\mathcal{R}_n({\bf r})}(A_n{\bm \alpha}-{\bm \beta}){\rm d}{{\bm \alpha}}{\rm d}{{\bm \beta}}.$$
Expanding both characteristic functions into Fourier series and using orthogonality gives
\begin{align*}
&\cL^{2d}(F_m\cap F_n) \\
&= \sum_{{\bm k},{\bm t}\in \mathbb{Z}^{d}} \hat\bone_{\cR_m(\br)}(\bk)
\hat\bone_{\cR_n(\br)}(\bt) \\
&\qquad \int e(\langle A_m^{T}{\bf k} + A_n^{T}{\bf t},{\bm \alpha}\rangle){\rm d}{\bm \alpha} \int e(-\langle {\bf k}+{\bf t},{\bm \beta}\rangle){\rm d}{\bm \beta} \\
&= \sum_{{\bf k}\in \mathbb{Z}^{d}} \hat\bone_{\cR_m(\br)}(\bk)
\hat\bone_{\cR_n(\br)}(-\bk) \int e(\langle (A_m - A_n)^{T}{\bf k},{\bm \alpha}\rangle){\rm d}{\bm \alpha}.
\end{align*}
As $\det(A_m - A_n) \ne 0$, the only contribution comes from ${\bf k}={\bm 0},$ so 
$$\mathcal{L}^{2d}(F_m\cap F_n) = \hat\bone_{\cR_m(\br)}(\bzero)
\hat\bone_{\cR_n(\br)}(\bzero) = \mathcal{L}^{2d}(F_m)\mathcal{L}^{2d}(F_n).$$
\end{proof}

By Lemmas \ref{BC lemma} and \ref{independent}, the set $\limsup F_n$ has full Lebesgue measure in $[0,1)^d \times [0,1)^d$. Fubini’s theorem now gives 
$$\mathcal{L}^{d}\big(\mathcal{T}_{{\bm \alpha}}(\mathcal{A}, {\bf r})\big)=1$$ 
for almost every vertical fibre. This completes the proof of Theorem \ref{t2}.

\section{Proof of Theorem \ref{t3}}

\subsection{Convergence part}

In this case,
\begin{equation}
\label{convergence}
\sum_{n=1}^{\infty}s_{n}({\bf r},f)<\infty.
\end{equation}
Let $i = i(n) \in \{ 1,2,\ldots,d \}$ be such that
\[
s_{n}({\bf r},f) = f(r_i(n))\prod_{j\colon r_{j}(n)
 \ge r_{i}(n)}\frac{r_j(n)}{r_i(n)}.
\]
We cover
$$E_n := \{{\bm \beta}\in[0,1)^{d}\colon A_n{\bm \alpha} \in {\bm \beta}+\mathcal{R}_n( {\bf r})\!\!\!\!\pmod1\}$$
by balls of radius $r_i$. The $j^{\mathrm{th}}$ side of the hyper-rectangle $\mathcal{R}(A_n{\bm \alpha}, {\bf r}(n))$ centred at $A_n{\bm \alpha}~ \!\!\!\!\pmod1$ has length $2r_{j}(n)$ which can be covered by 
\begin{equation*}
\begin{cases}
   1,  & \text{if  } r_{j}(n)< r_i(n) \\
    \frac{r_{j}(n)}{r_{i}(n)},  & \text{if  } r_{j}(n)\ge r_i(n)
\end{cases}
\end{equation*}
many intervals of radius $r_i(n).$ Consequently, if $\eta \ge r_i$ then
\[
\cH^f_\eta(E_n)
\le s_{n}({\bf r},f).
\]
By (\ref{convergence}) and the assumption that $|\br(n)|_\infty \to 0$ as $n \to \infty$, we now have 
$$\mathcal{H}^{f}\big(\mathcal{T}_{{\bm \alpha}}(\mathcal{A}, {\bf r})\big)\le \lim_{N\to\infty}\sum_{n=N}^{\infty}s_{n}({\bf r},f)=0.$$

\subsection{Divergence part}

We establish the divergence part in this subsection by applying Lemmas \ref{p:MDP} and \ref{t:weaken}.
To apply Lemma \ref{t:weaken}, we construct a lim sup set generated by a sequence of balls, which still has full Lebesgue measure for $\mathcal{L}^{d}$-almost every ${\bm \alpha}\in[0,1)^{d}$. We begin with this construction.

We know that
\begin{equation}
\label{divergence}
\sum_{n=1}^{\infty} s_n(\mathbf{r},f) = \infty.
\end{equation}
Since $f \preceq d$, for large $n$ we have $f(r_i(n)) \gg r_i(n)^d$, hence
\begin{equation}
\label{lower}
s_n(\mathbf{r},f) \gg \min_{1 \le i \le d} \left\{ r_i(n)^d \prod_{j\colon r_j(n)\ge r_i(n)} \frac{r_j(n)}{r_i(n)} \right\} = \prod_{i=1}^d r_i(n).
\end{equation}
Since the lim sup set $\mathcal{T}_{{\bm \alpha}}(\mathcal{A}, \bf r)$ is unaffected by the alteration of finitely many $n\in\N,$ we will henceforth assume that the inequality (\ref{lower}) holds for all $n\in\N.$

For each $n\in\N,$  we rearrange  $r_{1}(n),\ldots, r_{d}(n)$ as
$$r_{i_{1}}(n)\le r_{i_{2}}(n)\le \cdots\le r_{i_{d}}(n).$$
Note that the indices $i_{1},\ldots i_{d}$ depend on $n,$ although we suppress this dependence in the notation. One may readily verify that the inequalities
$$\prod_{j=1}^{d}r_{i_{j}}\le\cdots\le r_{i_{k}}(n)^{k} \prod_{j=k+1}^{d} r_{i_{j}} \le \cdots \le r_{i_{d}}(n)^{d}$$
hold for $k=1,2,\ldots,d$.
By (\ref{lower}), there exists $k = k(n) \in 
\{ 0,1,\ldots,d \} $ such that 
$$r_{i_{k}}(n)^{k} \prod_{j=k+1}^{d} r_{i_{j}} \ll s_{n}({\bf r},f) \ll r_{i_{{k+1}}}(n)^{k+1} \prod_{j=k+2}^{d}r_{i_{j}},$$
where $r_{i_{d+1}} = \infty$. Let
\begin{equation}
\label{b_n}
b_n = \left(\frac{s_{n}({\bf r},f)}{\prod_{j=k+1}^{d}r_{i_{j}}(n)}\right)^{1/k},
\end{equation}
and note that
$$r_{i_{k}} \ll b_n \ll r_{i_{k+1}}.$$
Define $\Psi =(\psi_{1},\ldots,\psi_d)$ by 
\begin{equation*}
\psi_{i_{j}}(n)=\begin{cases}
   r_{i_{j}}(n),   & \text{if  $j>k=k(n)$} \\
    b_n,  & \text{if  $j\le k=k(n)$}.
\end{cases}
\end{equation*}
Then
$$\psi_{i_{1}}(n)=\cdots=\psi_{i_{k}}(n)=b_n\ll\psi_{i_{k+1}}(n)\le\cdots\le \psi_{i_{d}}(n).$$

\begin{lemma}
\label{FM}
Let $\Psi$ be the function defined as above. Then, for $\mathcal{L}^{d}$-almost every ${\bm \alpha}\in[0,1)^{d}$,  
$$\mathcal{L}^{d}\big(\mathcal{T}_{{\bm \alpha}}(\mathcal{A}, \Psi)\big)=1.$$
\end{lemma}

\begin{proof}
By Theorem \ref{t2}, it suffices to show that 
$$\sum_{n=1}^{\infty}\prod_{j=1}^{d}\psi_j(n)=\infty.$$
Note that  
$$\prod_{j=1}^{d}\psi_j(n)=\prod_{j=1}^{k(n)}b_n\cdot\prod_{j=k(n)+1}^{d} r_{i_{j}}(n) = s_{n}({\bf r},f).$$
The conclusion now follows from (\ref{divergence}).
\end{proof}

Recall that 
\begin{align*}
\mathcal{T}_{{\bm \alpha}}(\mathcal{A},\Psi) = 
\bigcap_{N=1}^\infty
\bigcup_{n=N}^{\infty} G_n,
\end{align*}
where
\[
G_n =  
\{ {\bm \beta} \in[0,1)^{d}: A_n{\bm \alpha} \in {\bm \beta}+\mathcal{R}_n(\Psi)\!\!\!\!\pmod1  \}.
\]
For each $n$, the set $G_{n}$ has a finite cover $\mathcal{C}_n$ 
by balls of radius $C b_n$ centred at points in $G_n$, where $C = C_d$ is a large, positive constant. Consequently,
$$\mathcal{T}_{{\bm \alpha}}(\mathcal{A},\Psi) \subseteq \bigcap_{N=1}^{\infty} \bigcup_{n=N}^{\infty} \bigcup_{B\in\mathcal{C}_n}B.$$
By Lemma \ref{FM}, the right-hand side has full Lebesgue measure for $\mathcal{L}^{d}$-almost every ${\bm \alpha}\in[0,1)^{d}$. Moreover,  
$$\mathcal{T}_{{\bm \alpha}}(\mathcal{A},{\bf r}) \supseteq \bigcap_{N=1}^{\infty}\bigcup_{n=N}^{\infty}\bigcup_{B\in\mathcal{C}_n}B\cap E_n.$$
In light of Lemma \ref{t:weaken}, to establish that $$\mathcal{H}^{f}(\mathcal{T}_{\alpha}(\mathcal{A},{\bf r}))=\mathcal{H}^{f}([0,1)^{d})$$
holds for $\mathcal{L}^{d}$-almost every ${\bm \alpha}\in[0,1)^{d}$, it suffices to show that for every such covering $\mathcal{C}_n$ and every ball $B\in\mathcal{C}_n,$
\begin{equation}
\label{key}
\mathcal{H}_{\infty}^{f}(B\cap E_n) \gg \mathcal{L}^{d}(B)\asymp b_n^{d}.
\end{equation}

Since the radius of $B$ is $Cb_n$, the set $B\cap E_n$ contains an axis-aligned box $H_n = H_n(B)$
with side lengths 
$$r_{i_1}(n),\ldots, r_{i_{k}}(n),\underbrace{b_{n},\ldots,b_n}_{d-k \text{ terms}}.$$
The following estimate for the Hausdorff $f$-content of $H_n$ implies (\ref{key}), and will therefore complete the proof of Theorem \ref{t3}.

\begin{lemma}
Let $H_n$ be as above. Then 
$$\mathcal{H}_{\infty}^{f}(H_n) \gg_f b_n^{d}.$$
\end{lemma}

\begin{proof}
We will apply Lemma~\ref{p:MDP} to bound the Hausdorff $f$-content from below. Define a probability measure $\mu$ on $H_n$ by
$$\mu := \frac{\mathcal{L}^{d}\big|_{H_{n}}}{\mathcal{L}^{d}(H_n)}=\frac{\mathcal{L}^{d}\big|_{H_{n}}}{r_{i_{1}}(n)\cdots r_{i_{k}}(n) b_{n}^{d-k}}.$$
We estimate $\mu\big(B({\bf x},\rho)\big)$ for an arbitrary ball $B({\bf x},\rho)$ with ${\bf x}\in H_n$ and $\rho>0.$ For this proof only, we redefine $r_{i_{k+1}} = b_n$. We divide the analysis into three cases. 

\bigskip

\noindent \textbf{Case 1.} Suppose $0<2\rho\le r_{i_{1}}(n).$ Then
$$\mu(B({\bf x},\rho))\le \frac{(2\rho)^{d}}{r_{i_{1}}(n)\cdots r_{i_{k}}(n) b_{n}^{d-k}}.$$
Since $s \preceq f\preceq s+1$ for some $s\in\{0,1,\ldots,d-1\}$, we get
$$l\preceq f \text{ for all } l\le s, \text{ and } f\preceq l \text{ for all } l\ge s+1.$$ 
In particular, since $s+1\le d$ we have $f\preceq d,$ so
$$\frac{f(r_{i_1}(n))}{r_{i_{1}}(n)^{d}} \le \frac{f(2\rho)}{(2\rho)^{d}} \ \Longrightarrow \ (2\rho)^{d}\le 
\frac{r_{i_{1}}(n)^{d}f(2\rho)}{f(r_{i_{1}}(n))}.$$ 
Hence,
$$\mu(B({\bf x},\rho))\le \frac{f(2\rho)}{f(r_{i_{1}}(n))r_{i_{1}}(n)^{1-d}r_{i_{2}}(n)\cdots r_{i_{k}}(n) b_n^{d-k}}.$$

\bigskip

\noindent \textbf{Case 2.} Suppose $r_{i_{j}}(n)<2\rho\le r_{i_{j+1}}(n)$ for some $1\le j\le k \le d$.
Then either
$$d-j\preceq f \text{ or } f\preceq d-j.$$
We compute that
$$\mu(B({\bf x},\rho))\le\frac{r_{i_1}(n)\cdots r_{i_{j}}(n)(2\rho)^{d-j}}{r_{i_{1}}(n)\cdots r_{i_{k}}(n) b_{n}^{d-k}}=\frac{(2\rho)^{d-j}}{r_{i_{j+1}}(n)\cdots r_{i_{k}}(n) b_{n}^{d-k}}.$$
If $d-j\preceq f$, then from $r_{i_{j}}(n) < 2\rho$ we obtain 
$$\frac{(2\rho)^{d-j}}{f(2\rho)}\le \frac{r_{i_{j}}(n)^{d-j}}{f(r_{i_{j}}(n))} \ \Longrightarrow \ (2\rho)^{d-j} \le \frac{r_{i_{j}}(n)^{d-j}f(2\rho)}{f(r_{i_{j}}(n))},$$
whence
$$\mu(B({\bf x},\rho)) \le \frac{f(2\rho)}{f(r_{i_{j}}(n))r_{i_{j}}(n)^{j-d}
r_{i_{j+1}}(n) \cdots r_{i_{k}}(n) b_n^{d-k}}.$$
If instead $f\preceq d-j,$ then $2\rho\le r_{i_{j+1}}(n)$ gives
$$\frac{f(r_{i_{j+1}}(n))}{r_{i_{j+1}}(n)^{d-j}} \le \frac{f(2\rho)}{(2\rho)^{d-j}} \ \Longrightarrow \ (2\rho)^{d-j} \le \frac{r_{i_{j+1}}(n)^{d-j}f(2\rho)} {f(r_{i_{j+1}}(n))},$$
and thus 
$$\mu(B({\bf x},\rho))\le\frac{f(2\rho)}{f(r_{i_{j+1}}(n))r_{i_{j+1}}(n)^{j+1-d} r_{i_{j+2}}(n) \cdots r_{i_{k}}(n) b_n^{d-k}}.$$

\bigskip

In either of these first two cases,
$$\mu(B({\bf x},\rho))\le\frac{f(2\rho)}{\min\limits_{1\le j\le k+1}
\left\{f(r_{i_{j}}(n))r_{i_{j}}(n)^{j-d} r_{i_{j+1}}(n) \cdots r_{i_{k}}(n) b_n^{d-k}\right\}}.$$
For $1 \le j \le k$,
\begin{align*} 
&f(r_{i_j}(n))
r_{i_j}(n)^{j-d} r_{i_{j+1}}(n) \cdots r_{i_{k}}(n)b_n^{d-k}
\\
&= f(r_{i_j}(n)) \prod_{s=j+1}^{d} \frac{r_{i_{s}}(n)}{r_{i_j}(n)} \cdot \frac{(b_n)^{d-k}}{\prod_{s=k+1}^{d}r_{i_{s}}(n)}\\
&\ge s_n({\bf r},f)  \frac{(b_n)^{d-k}}{\prod_{s=k+1}^{d}r_{i_{s}}(n)}=b_n^{d},
\end{align*}
where the inequality uses the definition of $s_n({\bf r},f),$ and the last equality follows from (\ref{b_n}). For $j = k+1$,
\begin{align*}
&f(r_{i_j}(n))
r_{i_j}(n)^{j-d} r_{i_{j+1}}(n) \cdots r_{i_{k}}(n)b_n^{d-k}
\\
&= f(b_n) r_{i_j}(n) \cdots r_{i_k}(n) = f(b_n).
\end{align*}
The upshot is that, in Cases 1 and 2,
\[
\mu(B(\bx, \rho)) \le \frac{f(2 \rho)}{b_n^d}.
\]

\bigskip

\noindent
\textbf{Case 3.} Suppose $2\rho > b_n$. Then 
$$\mu(B(\bx,\rho)) \le 1 \le \frac{f(2\rho)}{f(b_n)} \ll \frac{f(2 \rho)}{b_n^d},$$
since $f\preceq d$. 

\bigskip

We have 
\[
\mu(B(\bx, \rho)) \ll \frac{f(2 \rho)}{b_n^d}
\]
in all three cases so, by Lemma \ref{p:MDP}, 
$$\mathcal{H}_{\infty}^{f}(H_n) \gg b_n^{d}.$$
\end{proof}

\subsection*{Acknowledgements} QZ was supported by NSFC 12201476 and CSC 202506950114. QZ also would like to take this opportunity to thank the number theory group at the University of Warwick --- in particular Sam Chow and Han Yu --- for the unreserved support and friendship they offered during his visit. We thank Victor Beresnevich for encouragement and comments.

\subsection*{Rights}

For the purpose of open access, SC has applied a Creative Commons Attribution (CC-BY) licence to any Author Accepted Manuscript version arising from this submission.

\end{document}